\documentclass[11pt,reqno,tbtags,a4paper]{amsart}
\usepackage{amssymb}
\usepackage{url}
\usepackage[square,numbers]{natbib}
\usepackage{graphics}

\bibpunct[, ]{[}{]}{;}{n}{,}{,}

\title[Estimating global subgraph counts by sampling]
{A note on estimating global subgraph counts by sampling}

\date{20 October, 2022}

\author{Svante Janson}
\thanks{SJ Supported by the Knut and Alice Wallenberg Foundation}
\address{Department of Mathematics, Uppsala University, PO Box 480,
SE-751~06 Uppsala, Sweden}
\email{svante.janson@math.uu.se}
\newcommand\urladdrx[1]{{\urladdr{\def~{{\tiny$\sim$}}#1}}}
\urladdrx{http://www2.math.uu.se/~svante/}

\author{Valentas Kurauskas}
\address{Faculty of Mathematics and Informatics, Vilnius University, Akademijos 4, 08412, Vilnius, Lithuania}
\email{valentas@gmail.com}

\subjclass[2020]{} 

\numberwithin{equation}{section}

\renewcommand\le{\leqslant}
\renewcommand\ge{\geqslant}

\allowdisplaybreaks

\theoremstyle{plain}
\newtheorem{theorem}{Theorem}

\newtheorem{corollary}[theorem]{Corollary}

\theoremstyle{definition}

\newtheorem{exampleqqq}[theorem]{Example}

\newtheorem{remarkqqq}[theorem]{Remark}

\theoremstyle{remark}

\newenvironment{romenumerate}[1][-10pt]{
\addtolength{\leftmargini}{#1}\begin{enumerate}
 \renewcommand{\labelenumi}{\textup{(\roman{enumi})}}%
 }{\end{enumerate}}

\newcounter{oldenumi}
{\setcounter{oldenumi}{\value{enumi}}
\begin{romenumerate} \setcounter{enumi}{\value{oldenumi}}}
{\end{romenumerate}}

\newcounter{thmenumerate}

\newcounter{xenumerate}   

\newcounter{pfcases}
\newcommand\pfcase[1]{\par\smallskip\noindent\refstepcounter{pfcases}%
\emph{Case \arabic{pfcases}: #1} \noindent}

\newcommand{\refT}[1]{Theorem~\ref{#1}}

\newcommand\marginal[1]{\marginpar[\raggedleft{\tiny #1}]{\raggedright{\tiny#1}}}
\newcommand\REM[1]{{\raggedright\texttt{[#1]}\par\marginal{XXX}}}
\newcommand\XREM[1]{\relax}

\begingroup
  \divide\count255 by 60
  \multiply\count255 by -60
  \advance\count255 by \time
  \ifnum \count255 < 10 \xdef\klockan{\the\count1.0\the\count255}
  \else\xdef\klockan{\the\count1.\the\count255}\fi
\endgroup

\newcommand\set[1]{\ensuremath{\{#1\}}}

\newcommand\bigpar[1]{\bigl(#1\bigr)}
\newcommand\Bigpar[1]{\Bigl(#1\Bigr)}

\newcommand\lrpar[1]{\left(#1\right)}
\newcommand\bigsqpar[1]{\bigl[#1\bigr]}

\newcommand\xsqpar[1]{\,[#1]}

\newcommand\bigabs[1]{\bigl\lvert#1\bigr\rvert}

\newcommand\biggabs[1]{\biggl\lvert#1\biggr\rvert}

\def\rompar(#1){\textup(#1\textup)}    

\def\xexp(#1){e^{#1}}

\newcommand\ntoo{\ensuremath{{n\to\infty}}}

\newcommand{\tend}{\longrightarrow}
\newcommand\dto{\overset{\mathrm{d}}{\tend}}

\newcounter{CC}
\newcounter{cc}

\newcommand\E{\operatorname{\mathbb E{}}}
\renewcommand\P{\operatorname{\mathbb P{}}}

\newcommand\ga{\alpha}

\newcommand\gD{\Delta}
\newcommand\gf{\varphi}

\newcommand\gl{\lambda}

\newcommand\gs{\sigma}

\renewcommand\phi{\xxx}  

\newcommand\indicx[1]{\boldsymbol1_{#1}} 
\newcommand\indic{\indicx}
\newcommand\indicq[1]{\indicx{\set{#1}}}

\newcommand\dd{\,\mathrm{d}}

\newcommand\Hom{\operatorname{Hom}}
\newcommand\Emb{\operatorname{Emb}}
\newcommand\homD{\hom_{\gD}}
\newcommand\homDa{\hom_{\gD,\bga}}
\newcommand\indicD[1]{\indicx{d_{#1}\ge\gD}}
\newcommand\indicDv{\indicD{v}}
\newcommand\bga{\boldsymbol{\ga}}
\newcommand\hX{\widehat{X}}
\newcommand\bX{\bar{X}}
\newcommand\vv{V}

\newcommand{\Holder}{H\"older}

\DeclareMathOperator{\emb}{emb}

\begin{document}

\begin{abstract} 
We give a simple proof of a generalization of an inequality for homomorphism
counts by Sidorenko \cite{Sidorenko}. A special case of our inequality says
that if $d_v$ denotes the degree of a vertex $v$ in a graph $G$ and $\homD(H, G)$ denotes
the number of homomorphisms from a connected graph $H$ on $h$ vertices to $G$ 
which map a particular vertex of $H$ to a vertex $v$ in $G$ with $d_v \ge \gD$, then
\begin{align*}
    \homD(H,G) \le \sum_{v\in G} d_v^{h-1}\indicDv.
\end{align*}
We use this inequality to study the minimum 
sample size
needed
to estimate the number of copies of $H$ in $G$ by sampling vertices of $G$
at random.
\end{abstract}

\maketitle

\section{Introduction}\label{S:intro}

This paper consists of two main parts. In Section~\ref{S:ineq} we present a simple proof of an inequality that generalizes Sidorenko's inequality on homomorphism counts. Our motivation for this result comes from an application to estimating global subgraph counts by sampling, which is discussed in Section~\ref{S:appl}.

\subsection*{Acknowledgement}
Much of this research was done during the
28th Nordic Congress of Mathematicians at
Aalto University, Helsinki, Finland in August 2022,
where the authors met. We thank the organizers for making this possible.

\section{The inequality}\label{S:ineq}

Let $H$ and $G$ be graphs.
We let $\Hom(H,G)$ denote the set of homomorphisms $H\to G$,
and $\hom(H,G):=|\Hom(H,G)|$ the number of them.

Sidorenko~\cite{Sidorenko} proved
the following theorem:
\begin{theorem}[Sidorenko, 1994]\label{thm.sidorenko}
For any connected graph $H$ on $h\ge1$ vertices and any graph $G$,
    \begin{align}
        \hom(H, G) \le \hom(K_{1, {h-1}}, G).
    \end{align}
\end{theorem}
In fact, Sidorenko \cite{Sidorenko} showed this
for trees $H$, but this is immediately equivalent to our formulation, since 
$\hom(H,G)\le \hom (T,G)$ for any spanning tree $T$ of $H$.

If $H$ is a rooted graph, with root $o$, and 
 $\gD\ge0$, we also define
 \begin{align}\label{homD}
   \homD(H,G):=\bigabs{\set{\gf\in\Hom(H,G):d_{\gf(o)}\ge\gD}},
 \end{align}
where $d_v$ here and below denotes the degree of a vertex $v$ in a graph.
(The graph will be clear from the context; in this section it is always $G$.)
We show the following extension of \refT{thm.sidorenko}.

\begin{theorem} \label{thm.homD}
For any connected rooted
graph $H$ on $h\ge1$ vertices, any graph $G$,
and any  $\gD\ge0$,
  \begin{align}\label{thomD}
    \homD(H,G) \le \sum_{v\in G} d_v^{h-1}\indicDv.
  \end{align}
\end{theorem}

Note that Sidorenko's theorem is the special case $\gD=0$
of our theorem.

We will use induction to prove a more general statement.
Let $\bga:=(\ga_w)_{w\in H}$ be a vector of non-negative real numbers
$\ga_w$
indexed by the vertices in $H$, and let
 \begin{align}\label{homDa}
   \homDa(H,G):=
\sum_{\gf\in\Hom(H,G)}\indicq{d_{\gf(o)}\ge\gD}\prod_{w\in H}d_{\gf(w)}^{\ga_w}.
 \end{align}
In particular, taking all $\ga_w=0$, we have $\hom_{\gD,\bf{0}}(H,G)=\homD(H,G)$.
Hence, \refT{thm.homD} is a special case of the following result.

\begin{theorem} \label{thm.homDa}
For any connected rooted graph $H$ on $h\ge1$ vertices, any graph $G$,
any  $\gD\ge0$, and any non-negative vector $\bga=(\ga_w)_{w\in H}$, 
\begin{align}
  \label{aalto}
    \homDa(H,G) \le \sum_{v\in G} d_v^{h-1+|\bga|}\indicDv,
\end{align}
where
\begin{align}\label{bga}
  |\bga|:=\sum_{w\in H} \ga_w.
\end{align}
\end{theorem}

\begin{proof}
To prove \eqref{aalto}, we use a double induction over the number of
vertices $h$ in $H$ and the number of non-root vertices $w$ such that the
weight $\ga_w>0$.
Hence, we may assume that \eqref{aalto} holds if we replace the
pair $(H,\bga)$ by a pair $(H',\bga')$ such that either
\begin{romenumerate}
    
\item 
$H'$  has fewer vertices than $H$, or
\item 
$H'$ has the same number of vertices as $H$, but there are
fewer non-root vertices $w\in H'$ with $\ga'_w>0$ than non-root $w\in H$
with $\ga_w>0$.  
\end{romenumerate}

The base case $h=1$ is trivial, since in this case \eqref{aalto} is an
identity.
To prove the induction step, we consider three cases separately.
\pfcase{$H$ has a leaf $w\neq o$ with $\ga_w=0$.}
Let $v$ be the unique neighbour of $w$ in $H$. 
Define $H':=H\setminus\set{w}$, and let $\ga'_v:=\ga_v+1$, and
$\ga'_u:=\ga_u$ for all other $u\in H'$.
Then $\homDa(H,G)=\hom_{\gD,\bga'}(H',G)$, and thus \eqref{aalto} follows by
  the induction hypothesis, since 
$H'$  has 
one vertex less
than $H$.

\pfcase{$H$ has (at least) two (distinct) non-roots $v$ and $w$ with
  $\ga_v,\ga_w>0$.} 
Here we use \Holder's inequality, 
in a way that is essentially the same as in \citet{Sidorenko}
(although he does it in a more general way).

By decomposing the sum in \eqref{homDa} according to the values of $\gf(v)$
and $\gf(w)$, we obtain
\begin{align}\label{magnus}
  \homDa(H,G) = \sum_{x,y\in G}\mu_{x,y} d_x^{\ga_v}d_y^{\ga_w},
\end{align}
for some numbers $\mu_{x,y}\ge0$ that do not depend on $\ga_v$ and $\ga_w$.
We regard the  numbers $\mu_{x,y}$ as a measure $\mu$ on the finite set
$V(G)\times V(G)$, and rewrite \eqref{magnus} as
\begin{align}
  \homDa(H,G)=\iint_{V(G)\times V(G)}  d_x^{\ga_v}d_y^{\ga_w} \dd\mu(x,y).
\end{align}
\Holder's inequality now yields
\begin{align}
&  \homDa(H,G)
\notag\\&\quad\le
\lrpar{\iint_{V(G)\times V(G)}  d_x^{\ga_v+\ga_w}
    \dd\mu(x,y)}^{\frac {\ga_v} {\ga_v+\ga_w}}
\lrpar{\iint_{V(G)\times V(G)}  d_y^{\ga_v+\ga_w}
    \dd\mu(x,y)}^{\frac {\ga_w} {\ga_v+\ga_w}}
\notag\\&\quad
    =\hom_{\gD,\bga'}(H,G)^{\frac {\ga_v} {\ga_v+\ga_w}}
    \hom_{\gD,\bga''}(H,G)^{\frac {\ga_w} {\ga_v+\ga_w}},
\end{align}
where 
\begin{align}
  \ga'_v&:=\ga_v+\ga_w, & \ga''_v&:=0,&&
\\
  \ga'_w&:=0, & \ga''_w&:=\ga_v+\ga_w,&&
\\
    \ga'_u&:=\ga''_u:=\ga_u \smash{\quad\text{for $u\not \in \{v,w\}$}}.
\end{align}
Hence \eqref{aalto} follows from the induction hypothesis, since both
$\bga'$ and $\bga''$ have one less non-root vertex with positive weight than
$\bga$. 

\pfcase{The remaining case.}
If none of the cases above applies, then
every non-root leaf has positive weight, and there is at most one non-root
vertex with positive weight. In particular, there is at most one non-root leaf.
If also $|V(H)|\ge2$, 
then $H$ must have exactly one non-root leaf, say $v$,
and thus $H$ is a path with end vertices $o$ and $v$. 
Furthermore, only $v$ and (possibly) the root  $o$ can have positive weight.
Thus
\begin{align}
  \homDa(H,G)=
\sum_{\gf\in\Hom(H,G)}\indicq{d_{\gf(o)}\ge\gD} d_{\gf(o)}^{\ga_o}d_{\gf(v)}^{\ga_v}.
\end{align}

In this case, we use a variant of an argument 
that has been used to show other inequalities
(see, e.g.,\ \cite[Theorems 43 and 236]{HLP} and \cite[Theorem 2.4]{Grimmett}).
We write, for $x\in G$,
\begin{align}
  f(x):=d_{x}^{\ga_o}\indic{d_{x}\ge\gD},
&&&
g(x):=
d_{x}^{\ga_v}.
\end{align}
Then both $f(x)$ and $g(x)$ are (weakly) increasing functions of $d_x$, and
thus, for all $x,y\in G$,
\begin{align}
  \bigpar{f(x)-f(y)}\bigpar{g(x)-g(y)}\ge0.
\end{align}
Consequently, using also the symmetry of $H$ interchanging $o$ and $v$,
\begin{align}
  0 
&\le 
\sum_{\gf\in\Hom(H,G)}\bigpar{f(\gf(o))-f(\gf(v))}\bigpar{g(\gf(o))-g(\gf(v))}
\notag\\&
=2\sum_{\gf\in\Hom(H,G)}f(\gf(o))g(\gf(o))
-2\sum_{\gf\in\Hom(H,G)}f(\gf(o))g(\gf(v))
\notag\\&
=2\hom_{\gD,\bga'}(H,G)-2\homDa(H,G),
\end{align}
where 
\begin{align}
  \ga'_o&:=\ga_o+\ga_v,
\\
\ga'_w&:=0\quad\text{for $w\neq o$}.
\end{align}
Thus 
$\homDa(H,G)\le\hom_{\gD,\bga'}(H,G)$, and thus \eqref{aalto} follows
by induction, 
since $\bga'$ has 
one less non-root vertex with positive weight than $\bga$.
\end{proof}

\begin{proof}[Proof of \refT{thm.homD}]
    As mentioned above, this is the special case $\bga=0$
of \refT{thm.homDa}.
\end{proof}

\section{Applications}  
\label{S:appl}

Let now $H$ be a fixed connected graph on $h$ vertices and $G$ an arbitrary (large) graph on $n$ vertices.
Let $\Emb(H,G)$ denote the set of embeddings (injective homomorphisms)
$H \to G$; we will be interested in estimating the number $\emb(H,G) :=
|\Emb(H,G)|$ by sampling a rather small number of vertices of $G$ and
exploring small neighbourhoods of them.

A similar problem for sequences of graphs with a weak local limit
has been studied in~\cite{kurauskas2015}; 
there a uniform integrability condition 
on the $(h-1)$th power of the
random vertex degree was used.
Uniform integrability of graph degrees or their powers is natural for
sequences of graphs,
and has been used both in theoretical work and in applications, see, e.g.,
\cite{andersson1998, bls2013}.
In our setting, 
we use instead the related 
\eqref{eq.Dhtail} below. The general problem
of estimating small subgraph counts in a given graph
has been considered by many authors for a variety of applications, 
see, e.g., the survey paper~\cite{ribeiro2022}.

To estimate $\emb(H,G)$, we may assume that $H$ is a rooted graph with a root $o$
($o$ can be chosen arbitrarily).
For a vertex $v\in G$, 
let $X(v)=X(H, G, v)$ 
denote the number of embeddings $\sigma \in \Emb(H,G)$ such
that $\sigma(o) = v$.
We may then estimate $\emb(H,G)$ from the numbers $X(v^*_i)$
for some
randomly sampled vertices $v^*_i$ in $G$.
However, since vertices of high
degree in $G$ may give outliers with exceptionally high numbers of such
embeddings, we use truncation in order to obtain our error bounds.

Choose an arbitrary rooted spanning tree $T$ of $H$ 
with the same root $o$.
Say that a vertex $u\in T$ 
is \emph{internal} if
$d^T_u > 1$ or $u=o$ (where $d_u^T$ denotes the degree of $u$ in $T$).
Denote by $i_T$ the
number of internal vertices in $T$.
Choose also a positive integer $\gD$.
For a vertex $v\in G$, 
let $\bX(v) = \bX(H, G, v, T, \gD)$
denote the number of embeddings $\sigma \in \Emb(H,G)$ such
that $\sigma(o) = v$ and $d_{\sigma(u)} < \gD$ for all internal vertices
$u\in T$. 

Let $N\ge1$ and let $v_1^*, \dots, v_N^*$ be drawn from $V(G)$ independently
and uniformly at random. 
Consider the following estimate for $n^{-1} \emb(H, G)$:
\begin{align}
    \hX_N := \frac 1 N \bigpar{\bX(v_1^*) + \dots + \bX(v_N^*)}.
\end{align}
Note that the random variable $\bX(v_1^*)$ is bounded, since for every $v\in
G$,
\begin{align}\label{bdd}
0 \le \bX (v) \le (\gD-1)^{h-1}.   
\end{align}
Its mean can be estimated by $\hX_N$, with an error that can be bounded
using, for example,  Hoeffding's bound, which we do to prove the following
result. 
\begin{theorem}\label{thm.hoeffd}
    Let $H$ be a connected graph on $h\ge1$ vertices with a rooted spanning tree $T$, 
    let $G$ be a graph on $n\ge1$ vertices,
    and
    let $D$ be the degree of a uniformly random vertex in $G$.
    Suppose a positive integer $\gD$
    and a non-negative $\lambda$ satisfy
\begin{equation}\label{eq.Dhtail}
    \E{}\bigsqpar{D^{h-1} \indicx{D \ge \gD}} \le \lambda.
\end{equation}
    Let $s > 0$ and $p \in (0,1]$. If
    \begin{equation}\label{eq.Nhoeffd}
        N \ge \frac {(\gD - 1)^{2(h-1)}} {2 s^2} \ln \frac 2 p,
    \end{equation}
    then
    \begin{align}\label{eq.th}
 \P\bigpar{\hX_N - s \le n^{-1} \emb(H,G) \le \hX_N + s + i_T \lambda } \ge 1-p.
    \end{align}
\end{theorem}
\begin{proof} 
We have 
    \begin{align}\label{magn}
        \emb(H, G) = \sum_{v \in V(G)} \bX(v) 
+ \biggabs{\bigcup_u \{\sigma \in \Emb(H, G): d_{\gs(u)} \ge \gD\}}, 
    \end{align}
    where the union is over the internal vertices of $T$.

    The first term on the right is equal to $n \E \bX(v_1^*)$. 
    By 
the union bound and    
Theorem~\ref{thm.homD} applied to each tree that can be obtained from
    $T$ by rerooting at an internal vertex,
    the second term is at most 
    $i_T n \E\xsqpar{ D^{h-1} \indicx{D \ge \gD}} \le i_T \lambda n$.
    Therefore
    \begin{equation}\label{eq.emb}
        \E \bX(v_1^*) \le  n^{-1} \emb(H, G) \le  \E \bX(v_1^*) + i_T \lambda.
    \end{equation}

    Hoeffding's classical inequality says
    that for a sum of independent random variables $X_1, \dots, X_N$ with
    values in $[0,1]$ and $\mu=N^{-1}\E(X_1+\dots+X_N)$ we have 
    \begin{align}\label{hoeffding}
\P(|N^{-1}(X_1 + \dots + X_N) - \mu| \ge t) \le 2\exp\left(- 2 N t^2\right).
    \end{align}
    The claim 
follows by applying this to the random variables 
$X_i := \bX(v_i^*)/(\gD-1)^{h-1}$ 
    and using (\ref{eq.emb}) (or using $\bX(v_i^*) = 0$ if $\Delta=1$). 
\end{proof}

In particular, we obtain by choosing
$s = \epsilon \E D^{h-1}$ in Theorem~\ref{thm.hoeffd}
the following corollary.
(Choosing $s$ in this way makes sense when
$n^{-1} \emb(H, G)$ is of the same order
as its upper bound $n^{-1} \hom(K_{1,{h-1}}, G) = \E D^{h-1}$, 
which often may be reasonable to expect in practice.)

\begin{corollary} \label{cor.hoeffd}
With notation as above, assume that
\eqref{eq.Dhtail} holds.
If  $\epsilon>0$,  $p\in(0,1]$, and
\begin{equation}\label{eq.hoeffd.practical}
 N \ge \frac {(\gD-1)^{2(h-1)}} { 2 \epsilon^2 (\E D^{h-1})^2} \ln \frac 2 p,
\end{equation}
then 
\begin{align}
    \P\bigpar{\hX_N - \epsilon \E D^{h-1} \le n^{-1} \emb(H,G) \le \hX_N +
  \epsilon \E D^{h-1}  + i_T \lambda} \ge 1-p.
\end{align}
\end{corollary}

If
we are able to draw edges $K_2$, wedges $P_3$ or other small subgraphs in $G$ uniformly at random,
we can estimate certain small subgraph densities using a much smaller sample
size 
using the following generalization of \refT{thm.hoeffd}.
Other authors have used other methods to obtain some practical results in estimating densities of graphs $H$ on $h \le 5$ vertices
using algorithms which sample random paths on up to 5 vertices as their first step,
see, e.g., \cite[Section~4.3]{ribeiro2022}. 

To state the generalization, 
let now $O$ be a
non-empty subgraph of $H$.
Let, as above,  $T$ be a spanning tree of $H$.
We declare that a vertex $v \in T$ is \emph{$O$-internal}
if either $v \in T \setminus V(O)$ and $d_v^T \ge 2$,
or $v \in O$ and there is an edge $uv \in T$ with $u \not \in O$.
We let $i_T^O$ denote the number of $O$-internal vertices in $T$.
Note that 
for $h \ge 2$ 
if $O$ consists of a single vertex $o$, this agrees
with the previous definition.

Assume $\emb(O,G) \ge 1$ and let $\nu \in \Emb(O,G)$. 
Let $\bX^O(\nu)$ denote the number of embeddings $\sigma \in \Emb(H, G)$
such that $\nu$ is the restriction of $\sigma$ to $O$,
and $d_{\gs(u)}<\gD$ for all $O$-internal vertices $u$.
Let $\nu_1^*, \dots, \nu_N^*$ be independent, uniformly random elements of $\Emb(O, G)$.
Consider the estimate
\begin{align}
    \hX_N^O := \frac 1 N \bigpar{\bX^O(\nu_1^*) + \dots + \bX^O(\nu_N^*)}.
\end{align}
Similarly as above we have:
\begin{theorem}\label{thm.hoeffdO}
    With notation and assumptions as above and in \refT{thm.hoeffd}, 
including \eqref{eq.Dhtail},
assume also $\emb(O, G) \ge 1$ and,
instead of \eqref{eq.Nhoeffd},
    \begin{equation}\label{eq.NhoeffdO}
        N \ge \frac {(\gD - 1)^{2(h-|V(O)|)}} {2 s^2} \ln \frac 2 p.
    \end{equation}
    Then
    \begin{align}\label{eq.thO}
        \P\Bigpar{\hX_N^O - s \le \frac{\emb(H,G)} {\emb(O,G)} \le \hX_N^O + s + \frac {i_T^O \lambda n} {\emb(O, G)} } \ge 1-p.
    \end{align}
\end{theorem}

\begin{proof}
    Using the same argument as for \eqref{eq.emb} we get by Theorem~\ref{thm.homD}:
    \begin{equation} \label{eq.embO}
        \emb(O, G)  \E \bX^O(\nu_1^*) \le  \emb(H, G) \le \emb(O, G)  \E \bX^O(\nu_1^*) +  i_T^O \lambda n.
    \end{equation}
    Note that $0 \le \bX^O(\nu) \le (\gD - 1)^{h - |V(O)|}$. To finish the proof, we again apply Hoeffding's inequality.
\end{proof}

 For a simple example, let $H=K_h$ and $T=O=K_{1,h-1}$; 
 then $|V(O)|=h$
and $i^O_T=0$. Thus the values of $\gD$ and $\gl$ are irrelevant; we take
$\gD = n$ and $\lambda=0$ so that \eqref{eq.Dhtail} holds.
\refT{thm.hoeffdO} shows that
$N \ge 2^{-1} s^{-2} \ln \frac 2 p$ suffices.
(In this simple case with $V(O)=V(H)$, 
this follows without invoking Theorem~\ref{thm.homD} in the proof.)

When we are only able to draw uniform embeddings of $O$ with $|V(O)| < h$, if (\ref{eq.Dhtail}) holds with non-trivial values of
$\gD$ and $\lambda$, we can apply Theorem~\ref{thm.hoeffdO}, for example, with $s=\epsilon \E D^{h-1}$ as in Corollary~\ref{cor.hoeffd}
not only to reduce $N$ but also to bound the steps needed to compute $\bX^O(\nu)$.

Finally, we note that Theorem~\ref{thm.homD} allows us to generalize (the
difficult part of) Theorem~2.1 of 
\cite{kurauskas2015}, with a simpler proof that does not require the local
weak convergence assumption of \cite{kurauskas2015}. 
\begin{theorem}
Let $H$ be a fixed connected graph on $h$ vertices, and pick an
arbitrary vertex $o \in H$ as its root. 
    Let $(G_n,n\in\{1,2,\dots,\})$ be a sequence of graphs.
Let $\vv_n$ be a uniformly random vertex from $V(G_n)$ and let
$D_n:=d_{\vv_n}$ be its degree in $G_n$. 

If\/ $X(H, G_n, \vv_n)$ converges in distribution to a random variable $X^*$ as \ntoo, and
$D_n^{h-1}$ is uniformly integrable, then 
\begin{align}\label{emma}
  |V(G_n)|^{-1} \emb(H, G_n) \to \E X^*.
\end{align}    
\end{theorem}
\begin{proof}
    Write $X_n := X(H, G_n, \vv_n)$.
    Clearly 
    \begin{align}\label{ika}
        \E X_n = |V(G_n)|^{-1} \emb(H,G_n).
    \end{align}
Hence 
we need to prove that
$\E X_n \to \E X^*$.
Since $X_n \dto X^*$, we have $\E X^* \le \liminf_{n\to\infty} \E X_n$
by Fatou's lemma.

To show the opposite inequality,
fix a rooted spanning tree $T$ of $H$ 
and a positive integer $\gD$, and write $\bX_n := \bX(H, G_n, \vv_n, T, \gD)$.
By \eqref{ika} and \eqref{magn} (with $G_n$ instead of $G$),
\refT{thm.homD} implies, as  above for \eqref{eq.emb},
\begin{align}\label{aw}
  \E X_n 
\le \E \bX_n + i_T \E{}[D_n^{h-1}\indicx{D_n\ge\gD}].
\end{align}
Furthermore, by \eqref{bdd}, writing $x\land y:=\min(x,y)$,
\begin{align}\label{bw}
  \bX_n 
\le 
X_n\land (\gD-1)^{h-1}.
\end{align}
Since $X_n$ converges in distribution to $X^*$,
we have 
\begin{align}\label{cw}
  \lim_{n\to\infty} \E \xsqpar{X_n \land (\gD-1)^{h-1}}
= \E \xsqpar{X^* \land (\gD-1)^{h-1}}
\le \E X^*.
\end{align}
    Write 
$\epsilon_\gD := 
\limsup_{n\to\infty}\E\xsqpar{ D_n^{h-1} \indicx{D_n \ge \gD}}$. 
    Then, by \eqref{aw}--\eqref{cw} and $i_T\le h$,
    \begin{align}
\limsup_{n \to \infty} \E X_n 
\le \limsup_{n\to\infty} \E\bX_n + h\epsilon_\gD \le \E X^* + h\epsilon_\gD.
    \end{align}
This holds for every $\gD>0$, 
and $\lim_{\gD \to \infty} \epsilon_\gD = 0$ by the uniform
integrability assumption; thus $\limsup_{n \to \infty} \E X_n \le \E X^*$,
which 
completes the proof.
\end{proof}

\section{Concluding remarks}

We extended Sidorenko's inequality and used it to derive 
bounds on 
the sample size
needed to estimate
the number of small subgraphs in a large graph
using only the weak assumption (\ref{eq.Dhtail}).

Like Hoeffding's bound, our estimate works for worst case graphs;
therefore the lower bound from Theorem~\ref{thm.hoeffd}
may be too pessimistic for specific real-world graphs.
Nevertheless it would be interesting
to better understand if our results and assumptions of type (\ref{eq.Dhtail})
can be useful in practice.

It would also be interesting to find an interpretation or applications 
of the general case of our inequality \eqref{aalto}.

\newcommand\AAP{\emph{Adv. Appl. Probab.} }
\newcommand\JAP{\emph{J. Appl. Probab.} }
\newcommand\JAMS{\emph{J. \AMS} }
\newcommand\MAMS{\emph{Memoirs \AMS} }
\newcommand\PAMS{\emph{Proc. \AMS} }
\newcommand\TAMS{\emph{Trans. \AMS} }
\newcommand\AnnMS{\emph{Ann. Math. Statist.} }
\newcommand\AnnPr{\emph{Ann. Probab.} }
\newcommand\CPC{\emph{Combin. Probab. Comput.} }
\newcommand\JMAA{\emph{J. Math. Anal. Appl.} }
\newcommand\RSA{\emph{Random Structures Algorithms} }
\newcommand\DMTCS{\jour{Discr. Math. Theor. Comput. Sci.} }

\newcommand\AMS{Amer. Math. Soc.}
\newcommand\Springer{Springer-Verlag}
\newcommand\Wiley{Wiley}

\newcommand\vol{\textbf}
\newcommand\jour{\emph}
\newcommand\book{\emph}
\newcommand\inbook{\emph}
\def\no#1#2,{\unskip#2, no. #1,} 
\newcommand\toappear{\unskip, to appear}

\newcommand\arxiv[1]{\texttt{arXiv}:#1}
\newcommand\arXiv{\arxiv}

\newcommand\xand{and }
\renewcommand\xand{\& }

\def\nobibitem#1\par{}

\newpage

\appendix

\section{A note on real-world experiments}

During the preparation of this paper we tested Corollary~\ref{cor.hoeffd} with some real-world degree distributions. 
For our first experiment, we considered survey data on self-reported human contact count distributions during a single day in the USA in 4 COVID-19 pandemic waves \cite{feehanmahmud2021}. For our second experiment, we considered degree distributions of more than 500 empirical networks of various types and sizes made available as part of the supporting code of \cite{broidoclauset2019}.

Although the lower bounds in our first experiment
seem to be interesting for further exploration, we believe that the survey sizes in \cite{feehanmahmud2021} (several thousand respondents in each of the COVID-19 waves) were too small to determine if there exists a practically useful choice of $\lambda$ and $\gD$ in (\ref{eq.Dhtail}), even for $h=3$.
In the second case for 75\% of degree distributions we got a lower bound for $N$ exceeding the underlying network size. 

We believe that in the first case establishing a better understanding on the degree tails, simply collecting more data or applying the adaptive mean estimation methods \cite{dklr2000, empiricalbernstein} might help. In the second case, since the full network data is available, the methods mentioned in \cite{ribeiro2022} and Theorem~\ref{thm.hoeffdO} seem to be more suitable.

The code and the results of our experiments are available at
\\\url{https://github.com/valentas-kurauskas/subgraph-counts-hoeffding}.

\end{document}